\scrollmode

\documentclass[11pt]{amsart}
\usepackage{amssymb}
\usepackage[all]{xy}

\theoremstyle{plain}
\newtheorem{thm}{Theorem}

\newtheorem*{thmA}{Theorem A}
\newtheorem*{thmBp}{Theorem B$^\prime$}
\newtheorem*{thmB}{Theorem B}
\newtheorem*{thmC}{Theorem C}
\newtheorem*{thmD}{Theorem D}
\newtheorem*{corE}{Corollary E}
\newtheorem{lem}[thm]{Lemma}
\newtheorem{cor}[thm]{Corollary}

\newtheorem{prop}[thm]{Proposition}
\newtheorem{ques}{Question}

\theoremstyle{remark}

\newtheorem{rem}[thm]{Remark}

\numberwithin{thm}{section}
\numberwithin{equation}{section}

\theoremstyle{definition}

\newcommand{\ab}{\mathrm{ab}}

\newcommand{\caR}{\mathcal{R}}
\newcommand{\triv}{\{1\}}

\newcommand{\ZbG}{\Z_\bullet[G]}
\newcommand{\ZG}{\Z[G]}

\newcommand{\tc}{\mathbf{tc}}
\newcommand{\tk}{\mathbf{tk}}
\newcommand{\tf}{\mathrm{tf}}
\newcommand{\gtf}{\mathrm{gtf}}
\newcommand{\scd}{\mathrm{scd}}
\newcommand{\cl}{\mathrm{cl}}
\newcommand{\Aut}{\mathrm{Aut}}

\DeclareMathOperator{\ccd}{cd}

\newcommand{\Z}{\mathbb{Z}}

\newcommand{\Q}{\mathbb{Q}}

\newcommand{\caO}{\mathcal{O}}
\newcommand{\eucl}{\mathfrak{cl}}
\newcommand{\caV}{\mathcal{V}}
\newcommand{\Gal}{\mathrm{Gal}}

\newcommand{\boX}{\mathbf{X}}
\newcommand{\boY}{\mathbf{Y}}

\newcommand{\cMF}{\mathbf{cMF}}

\newcommand{\boc}{\mathbf{c}}
\newcommand{\bok}{\mathbf{k}}

\newcommand{\boB}{\mathbf{B}}
\newcommand{\boZ}{\mathbf{Z}}

\newcommand{\Ab}{\mathbf{Ab}}

\newcommand{\caM}{\mathcal{M}}

\newcommand{\kernel}{\mathrm{ker}}
\newcommand{\coker}{\mathrm{coker}}
\newcommand{\image}{\mathrm{im}}
\newcommand{\iid}{\mathrm{id}}
\newcommand{\Hom}{\mathrm{Hom}}
\newcommand{\Ext}{\mathrm{Ext}}

\newcommand{\hH}{\widehat{\mathrm{H}}}

\newcommand{\Zmod}{{}_{\Z}\mathbf{mod}}

\newcommand{\argu}{\hbox to 7truept{\hrulefill}}

\begin{document}

\title[A group theoretical version of Hilbert's theorem 90]{A group theoretical 
version of\\ Hilbert's theorem 90}
\author{C. Quadrelli and Th. Weigel}
\date{\today}
\address{C. Quadrelli, Th. Weigel\\ Dipartimento di Matematica e Applicazioni\\
Universit\`a di Milano-Bicocca\\
Ed.~U5, Via R.Cozzi 53\\
20125 Milano, Italy}
\email{c.quadrelli@campus.unimib.it, thomas.weigel@unimib.it}

\begin{abstract}
It is shown that for a normal subgroup $N$ of a group $G$, $G/N$ cyclic,
the kernel of the map $N^{\ab}\to G^{\ab}$ satisfies
the classical Hilbert 90 property (cf. Thm~A).
As a consequence, if $G$ is finitely generated, $|G:N|<\infty$, and all
abelian groups $H^{\ab}$, $N\subseteq H\subseteq G$, are torsion free,
then $N^{\ab}$ must be a pseudo permutation module for $G/N$
(cf. Thm.~B). 
From Theorem~A one also deduces a non-trivial relation between the order of the transfer kernel and co-kernel
which determines the Hilbert-Suzuki multiplier (cf. Thm.~C).
Translated into a number theoretic context one obtains a strong form of Hilbert's theorem 94 (Thm.~\ref{thm:h94}). 
In case that $G$ is finitely generated and $N$ has prime index $p$ in $G$ there holds a
``generalized Schreier formula'' involving the torsion free ranks of $G$ and $N$ and 
the ratio of the order of the transfer kernel and co-kernel (cf. Thm.~D).
\end{abstract}


\dedicatory{To the memory of K.W.~Gruenberg}

\maketitle

\section{Introduction}
\label{s:hil90}
Certainly, Hilbert's theorem 90 is
one of the first fundamental results in modern algebraic number
theory. In its original form one may state it as follows (cf. \cite[Thm.~90]{hil:alg}): If $E/F$ is a finite Galois extension with cyclic Galois group $G=\langle \sigma\rangle$, then any element $x$ in the kernel of the map
$N_G\colon L^\times\to K^\times$, $N_G(z)=\prod_{g\in G} g(z)$, $z\in L^\times$, can be written as $x=\sigma(y)y^{-1}$ for some $y\in L^\times$, i.e., in more sophisticated terms
$\hH^{-1}(G,L^\times)=0$.
In this note we want to establish an analogue of Hilbert's theorem~90
in a group theoretical context and to discuss some of its immediate consequences.
A (closed) normal subgroup $N$ of a (pro-$p$) group $G$ will be said to be {\it co-cyclic},
if $G/N$ is a cyclic group. For any (closed) subgroup $U$ of $G$ 
\begin{equation}
\label{eq:ab}
U^{\ab}=U/\cl([U,U])
\end{equation}
will denote the {\it maximal abelian (pro-$p$) quotient} of $U$. In particular, one has
a canonical map $t_{U,G}\colon U^{\ab}\to G^{\ab}$ induced by inclusion.
Elementary commutator calculus implies the following 
group theoretical version of Hilbert's theorem~90
(cf. Lemma~\ref{lem:cycsec}(b), \eqref{eq:seccoh1}, Prop.~\ref{prop:h90}).

\begin{thmA}
Let $G$ be a (pro-$p$) group, let $N$ be a co-cyclic (closed) normal subgroup of $G$, and
let $s\in G$ be an element such that $G=S\,N$ for $S=\cl(\langle s\rangle)$.
Then $\kernel(t_{N,G})=(s-1)\cdot N^{\ab}=[s,N]\cdot [N,N]/[N,N]$.
In particular, $\boc_1(G/N,\Ab)=0$.
\end{thmA}

Let $G$ be a finite group, and let $\Z_\bullet$ be either $\Z$ or $\Z_p$. 
A left $\ZbG$-module $M$ is said to be a
{\it $\ZbG$-lattice}, if $M$ is finitely generated and $M$ - considered as $\Z_\bullet$-module - is torsion free. 
A $\ZbG$-lattice $M$ is said to be a {\it $\ZbG$-permutation module}, if there exists
a finite left $G$-set $\Omega$ such that $M$ is isomorphic to $\Z_\bullet[\Omega]$, the free
$\Z_\bullet$-module spanned by the elements of $\Omega$. Moreover, a $\ZbG$-lattice is said to be a
{\it pseudo $\ZbG$-permutation module}, if it is isomorphic to a direct summand of some 
$\ZbG$-permutation module.
From Theorem~A one concludes the following.

\begin{thmB}
Let $G$ be a finitely generated group, and let $N$ be a co-cyclic normal subgroup of finite index in $G$
with the property that for every subgroup $H$ of $G$, $N\subseteq H\subseteq G$,
the abelian group $H^{\ab}$ is torsion free. Then $N^{\ab}$ is a pseudo
$\Z[G/N]$-permutation module.
\end{thmB}

The prefix ``pseudo'' arises from the phenomenon that for a finite group $G$
direct summands of $\ZG$-permutation modules are not necessarily
$\ZG$-permutation modules. This phenomenon does not occur
if $G$ is a finite $p$-group and $\Z_\bullet=\Z_p$, i.e., 
the pro-$p$ analogue of Theorem~B has the following form.

\begin{thmBp}
Let $G$ be a finitely generated pro-$p$ group, and let $N$
be a co-cyclic open normal subgroup of $G$
with the property that for every open subgroup $H$ of $G$, $N\subseteq H\subseteq G$,
the abelian pro-$p$ group $H^{\ab}$ is torsion free. Then $N^{\ab}$ is a $\Z_p[G/N]$-permutation module.
\end{thmBp}

If $U$ is a (closed) normal subgroup of finite index in a (pro-$p$) group $G$ the {\it transfer}
\begin{equation}
\label{eq:trans}
i_{G,U}\colon G^{\ab}\longrightarrow U^{\ab}
\end{equation}
from $G$ to $U$ is given by \[i_{G,U}(g\,\cl([G,G]))=\prod_{r\in\caR} r g r^{-1}\cl([U,U])\]
where $\caR\subseteq G$ is a set of representatives for the right $U$-cosets,
i.e., $G=\bigsqcup_{r\in\caR} r\,U$, where $\sqcup$ denotes "disjoint union"
(cf. \cite[Chap.~10]{rob:grp}). For $G$ and $U$ as above  
\begin{align}
\tk(G/U)&=\kernel(i_{G,U})\label{eq:tk}
\intertext{is called the {\it transfer kernel}, and}
\tc(G/U)&=\coker\left(i_{G/U}^\circ\colon G^{\ab}\longrightarrow (U^{\ab})^{G/U}\right).\label{eq:tc}
\end{align}
the {\it transfer cokernel}, where $\argu^{G/U}$ are the {\it $G/U$-invariants} of a left $\Z_\bullet[G/U]$-module.
The order of the transfer kernel $\tk(G/U)$ for a finite group $G$ and $[G,G]\subseteq U$ has been subject
of intensive investigations (cf. \cite{kala:bord}, \cite{kala:tk}, \cite{kala:tk2}) which were stimulated by 
Hilbert's theorem 94 (cf. \cite[Thm.~94]{hil:alg}) and Ph. Furtw\"angler's solution
of Hilbert's "principal ideal conjeture" (cf. \cite{furt:tvt}). Certainly, the most celebrated theorem in this context is due to
H.~Suzuki (cf. \cite{suz:h94}) which states that if $G$ is finite and $G/U$ is abelian,
the order of $\tk(G/U)$ must be a multiple of the order of $G/U$, i.e.,
there exists a positive integer $s_{G,U}$ such that $|\tk(G/U)|=s_{G,U}\cdot |G:U|$.
However, the question which remains is the size of the {\it Hilbert-Suzuki multiplier} $s_{G,U}$.
If $U$ is co-cyclic in $G$, then one may answer the latter question using the group
theoretical version of Hilbert's theorem~90.

\begin{thmC}
Let $G$ be FAb (pro-$p$) group, and let $N$ be a co-cyclic (closed) normal subgroup of finite index.
Then one has
\begin{equation}
\label{eq:tri}
|\tk(G/N)|=|G/N|\cdot |\tc(G/N)|,
\end{equation}
i.e., if $G$ is finite, then $s_{G,N}=|\tc(G/N)|$.
\end{thmC}

A finitely generated (pro-$p$) group $G$ is said to be\footnote{This abbreviation stands for {\it finite abelianizations}.} {\it FAb}, if for any (closed) subgroup $U$ of finite index in $G$ the group $U^{\ab}$ is finite\footnote{A pro-$p$ group $G$
satisfying $|G^{\ab}|<\infty$ must be finitely generated.}.
Theorem~C can be used to deduce a strong form of Hilbert's theorem 94 
stating that for finite cyclic unramified extensions of number fields the order
of the {\it capitulation kernel} is the product of the order of the {\it capitulation cokernel} times the degree
(cf. Thm.~\ref{thm:h94}). So far the capitulation cokernel has not found much attraction in algebraic number theory 
(cf. \cite{ngu:cocap}). This fact might be the reason why this stronger form 
of Hilbert's theorem 94 has not been established before.

If $G$ is a finitely generated (pro-$p$) group,
and $N$ is a (closed) normal subgroup of finite index, we call
\begin{equation}
\label{eq:trat}
\rho(G/N)=\frac{|\tk(G/N)|}{|\tc(G/N)|}\in\Q^\times
\end{equation}
the {\it transfer ratio} of $N$ in $G$. E.g., Theorem~C implies that if $G$
is a finite group and $N$ is co-cyclic, then $\rho(G/N)=|G:N|$.
Let $\Q_\bullet$ denote the quotient field of $\Z_\bullet$.
We will call the non-negative integer
\begin{equation}
\label{eq:tfr}
\tf(G)=\dim_{\Q_\bullet}(G^{\ab}\otimes_{\Z_\bullet}\Q_\bullet)
\end{equation}
the {\it torsion-free rank} of $G$, or for short the {\it tf-rank} of $G$.
One has the following ``generalized Schreier formula'' involving the transfer ratio.

\begin{thmD}
Let $G$ be a finitely generated (pro-$p$) group,
and let $U$ be a (closed) subgroup of prime index $p$.
Then one has
\begin{equation}
\label{eq:ratfor}
\tf(U)=p\cdot\tf(G)+(1-p)(1-\log_p(\rho(G/U))),
\end{equation}
where $\log_p(\argu)$ denotes the logarithm to the base $p$.
\end{thmD}

A finitely generated (pro-$p$) group $G$ is said to be of {\it global tf-rank} $\gtf(G)\geq 0$, if
one has $\tf(U)=\tf(G)$ for any (closed) subgroup $U$ which is of finite index in $G$.
E.g., a finitely generated (pro-$p$) group $G$ is of global tf-rank $0$ if, and only if,
$G$ is FAb. Using Theorem~D one concludes easily that a finitely generated (pro-$p$) group
$G$, which is of global tf-rank, must satisfy
\begin{equation}
\label{eq:gtf}
\gtf(G)=1-\log_p(\rho(U/V)),
\end{equation}
for any pair of (closed) subgroups $U,V$ of $G$, $V\subseteq U$, $U$ is of finite index in $G$, and
$V$ is normal in $U$ satisfying $|U:V|=p$. The following result generalizes the
well known fact that a pro-$p$ group which is FAb and of strict cohomological dimension less or equal to $2$
must be the trivial group (cf. Cor.~\ref{cor:brum}).

\begin{corE}
Let $G$ be a pro-$p$ group of global tf-rank satisfying $\scd_p(G)\leq 2$. Then either $G=\triv$ or $\gtf(G)=1$.
\end{corE}
 
The proof of Theorem~A is elementary, while
the proofs of Theorem B-D are easy but require some more
sophisticated ideas from the theory of cohomological Mackey functors
as well as some facts from the representation theory and cohomology theory
of cyclic groups. Nevertheless, in neither of the statements the reader will find any
trace of these sophisticated theories.

\vskip9pt
\noindent
{\bf Notation:}
As discrete groups and pro-$p$ groups behave quite similar, we will deal with
these two cases simultaneously.
We just add in parenthesis (...) the additional hypothesis or conclusions
in the case of pro-$p$ groups. By $\cl(\argu)$ we denote the closure
operation in a topological space. Moreover, $\Z_\bullet$ will denote the ring of integers $\Z$
in the case of discrete groups, and the ring of $p$-adic integers $\Z_p$ in the case
of pro-$p$ groups.


\section{Commutator calculus}
\label{s:commcal}
Let $G$ be a (pro-$p$) group. For two elements $x,y\in G$ we denote by
\begin{equation}
\label{eq:comm1}
[x,y]=x\,y\,x^{-1}y^{-1}
\end{equation}
their {\it commutator}, while for two (closed) subgroups $U$ and $V$ of $G$ we put
\begin{equation}
\label{eq:comm2}
[U,V]=\cl(\langle\,[x,y]\mid x,y\in U\,\rangle).
\end{equation}
From the commutator calculus in groups
one deduces the following.

\begin{lem}
\label{lem:cycsec}
Let $N$ be a co-cyclic (closed) normal subgroup of a (pro-$p$) group $G$,
and let $s\in G$ be an element such that $G=S\,N$
for $S=\cl(\langle s\rangle)$. Then
\begin{itemize}
\item[(a)] $[G,G]=[S,N]\,[N,N]$.
\item[(b)] $[G,G]=[s,N]\cdot [N,N]$, where
$[s,N]=\{\,[s,v]\mid v\in N\,\}$.
\end{itemize}
\end{lem}

\begin{proof}
As $G/N$ is abelian, one has that $[G,G]\subseteq N$,
i.e., $[G,G]/[N,N]$ is abelian.
From the
commutator identities (cf. \cite{rob:grp})
\begin{equation}
\label{eq:commid1}
[ab,c]={}^a[b,c]\cdot[a,c]\qquad\text{and}\qquad
[a,bc]=[a,b]\cdot {}^b[a,c],\qquad a,b,c\in G,
\end{equation}
one concludes that for $t_1,t_2\in S$, $v_1,v_2\in N$ one has
\begin{equation}
\label{eq:commid2}
[t_1v_1,t_2v_2]={}^{t_1}[v_1,t_2]\cdot {}^{t_1t_2}[v_1,v_2]\cdot
{}^{t_2}[t_1,v_2]\in [^{t_1}v_1,t_2]\cdot[t_1,{}^{t_2}v_2]\,[N,N].
\end{equation}
Here we used the fact that $[t_1,t_2] \in [S,S]=\triv$. This yields (a).
From the second identity in \eqref{eq:commid1} one concludes that
$[s,N]\,[N,N]/[N,N]\subseteq [G,G]/[N,N]$
is a (closed) subgroup satisfying 
\begin{equation}
\label{eq:commid4}
[s,v_1]\, [N,N] \cdot [s,v_2]\,[N,N]=[s,v_1v_2]\,[N,N].
\end{equation}
As $[v,t]=[t,v]^{-1}$ for $t\in S$, $v\in N$, 
in order to prove (b) it suffices to show that
$[t,v]\in [s,N]\,[N,N]$ for all $t\in S$, $v\in N$. From the first identity in
\eqref{eq:commid1} one concludes that
\begin{equation}
\label{eq:commid5}
[s^k,v] = [s^{k-1},{}^sv]\cdot [s,v]
\end{equation}
and, by induction, that $[s^k,v]\in [s,N]\,[N,N]$ for all $k \geq 0$ and $v\in N$.
By the first identity in \eqref{eq:commid1}, one has
\begin{equation}
\label{eq:commid6}
[s^{-1},v] = [s,{}^{s^{-1}}v]^{-1}\in [s,N]\,[N,N]
\end{equation}
and 
\begin{equation}
\label{eq:commid7}
[s^{-k},v] = [s^{1-k},{}^{s^{-1}}v]\cdot [s^{-1},v]
\end{equation} 
for all $k\geq 1$ and $v\in V$.  
Hence, by induction, $[s^k,N]\subseteq [s,N]\,[N,N]$
for all $k\in\Z$. This yields the claim if $G$ is discrete.
If $G$ is pro-$p$, then $(\bigcup_{k\in\Z}[s^k,N]) [N,N]/[N,N]$
is dense in $[S,N]\,[N,N]/[N,N]$, and $[s,N]\,[N,N]/[N,N]$
is closed. This yields the claim in the pro-$p$ case.
\end{proof}


\section{Cohomological Mackey functors}
\label{s:cMF}
Let $G$ be a finite group. A {\it Mackey system} of $G$ is a set of subgroups
of $G$ which is closed under conjugation and intersection, e.g.,
the set $G^\sharp$ of all subgroups of $G$, the set $G^\natural$ of all
normal subgroups of $G$ and $G^\circ=\{\triv,G\}$ are Mackey systems of
$G$. Let $\caM$ be a Mackey system of $G$.
A {\it cohomological $\caM$-Mackey functor} $\boX$ with values in the
category of abelian groups is a collection of abelian groups $\boX_U$,
$U\in\caM$, together with a collection of group homomorphisms
\begin{equation}
\label{eq:cMFdef}
i_{U,V}^\boX\colon \boX_U\to\boX_V,\quad
t_{V,U}^\boX\colon\boX_V\to\boX_U,\qquad
c_{g,U}^\boX\colon \boX_U\to\boX_{{}^gU},
\end{equation}
for $U,V\in\caM$, $V\subseteq U$ and $g\in G$,
which satisfy certain identities (cf. \cite[\S 3.1]{blth:cyc}).
Moreover, if $V$ is normal in $U$, then $\boX_V$ carries naturally
the structure of a left $\Z[U/V]$-module, and - considering
$\boX_U$ as trivial $\Z[U/V]$-module - the mappings
$i_{U,V}^\boX$ and $t_{V,U}^\boX$ are homomorphisms of
left $\Z[U/V]$-modules. 
A homomorphism of cohomological $\caM$-Mackey functors
$\eta\colon\boX\to\boY$ is a collection $(\eta_U)_{U\in\caM}$
of group homomorphism $\eta_U\colon\boX_U\to\boY_U$
which commute with the maps defined in \eqref{eq:cMFdef}.
The category of cohomological $\caM$-Mackey functors $\cMF_\caM(\Zmod)$
with values in the category of abelian groups
coincides with the category of contravariant additive
functors on some additive category of $\Z[G]$-permutation modules.
Therefore, it is an abelian category.
For further details the reader may wish to consult
\cite{bouc:mc}, \cite{blth:cyc} or \cite{webb:mc}.


\subsection{The cohomological Mackey functor $\Ab$}
\label{ss:Ab}
Let $G$ be a (pro-$p$) group, and let $N$ be a (closed) normal subgroup of finite index.
Then $\Ab$ is the cohomological $(G/N)^\sharp$-Mackey functor
with values in the category of abelian groups (resp. abelian pro-$p$ groups)
given by $\Ab_U=U^{\ab}$. Moreover, for $N\subseteq V\subseteq U\subseteq G$,
the map $t_{V,U}^{\Ab}\colon V^{\ab}\to U^{\ab}$ is just the canonical map,
while $i_{U,V}^{\Ab}\colon U^{\ab}\to V^{\ab}$ coincides with the transfer.
For further details see \cite[\S3.1]{thw:fratt}  and \cite[\S 3.8]{thw:prdim}.


\subsection{Section cohomology groups}
\label{ss:seccoh}
Let $G$ be a finite group,
and let $\boX$ be a cohomological $G^\circ$-Mackey functor, where $G^\circ=\{\triv,G\}$. 
We denote by
\begin{equation}
\label{eq:seccoh1}
\begin{aligned}
\boc_0(G,\boX)&=\coker(t^\boX_{\triv,G}),& \bok^0(G,\boX)&=\kernel(i^\boX_{G,\triv}),\\
\boc_1(G,\boX)&=\kernel(t^\boX_{\triv,G})/\omega_{G}\cdot \boX_{\triv},&\bok^1(G,\boX)&=\boX_{\triv}^G/\image(i^\boX_{G,\triv}),
\end{aligned}
\end{equation}
where $\omega_G=\kernel(\ZG\to\Z)$ denotes the augmentation ideal, its {\it section cohomology groups}.
One has a (canonical) 6-term exact sequence
\begin{equation}
\label{eq:6term}
\xymatrix{
0\ar[r]&\boc_1(G,\boX)\ar[r]&\hH^{-1}(G,\boX_{\triv})\ar[r]&
\bok^0(G,\boX)\ar[d]\\
0&\bok^1(G,\boX)\ar[l]&\hH^{0}(G,\boX_{\triv})\ar[l]&\boc^0(G,\boX)\ar[l]\\
}
\end{equation}
of abelian groups (cf. \cite[Prop.~4.1(a)]{blth:cyc}, \cite[Prop.~2.1]{thw:fratt}).  
Here $\hH^k(G,\argu)$ denotes Tate cohomology (cf. \cite[Chap.~VI.4]{brown:coh}).
From the identity $t^\boX_{\triv,G}\circ i^\boX_{G,\triv}=|G|\cdot\iid_{\boX_G}$
(cf. \cite[\S 3.1]{blth:cyc}) follows that $|G|\cdot\bok^0(G,\boX)=0$
and $|G|\cdot\boc_0(G,\boX)=0$.
In particular, if $\boX_G$ is torsion free, then $\bok^0(G,\boX)=0$. 
Since $|G|\cdot\hH^{k}(G,\boX_{\triv})=0$ for $k\in\Z$,
one concludes from \eqref{eq:6term} that $\boc_1(G,\boX)$ and $\bok^1(G,\boX)$ 
satisfy the same relation. In particular, if $\boX_G$ and $\boX_{\triv}$
are finitely generated $\Z$-modules (resp. $\Z_p$-modules), then
all the groups defined in \eqref{eq:seccoh1} are finite.
Applied to co-cyclic normal subgroups one has obtains the following.

\begin{prop}
\label{prop:h90}
Let $G$ be a (pro-$p$) group, and let 
$N$ be a (closed) co-cyclic normal subgroup of finite index in $G$.
Then
\begin{equation}
\label{eq:seccoh2}
\begin{aligned}
\boc_0(G/N,\Ab)&\simeq G/N,& \bok^0(G/N,\Ab)&=\tk(G/N),\\
\boc_1(G/N,\Ab)&=0,&\bok^1(G/N,\Ab)&=\tc(G/N).
\end{aligned}
\end{equation}
\end{prop}

\begin{proof}
The statements for $\bok^0(G/N,\Ab)$ and $\bok^1(G/N,\Ab)$
are just the definition, while $\boc_0(G/N,\Ab)\simeq G/N$
follows from the fact that $[G,G]\subseteq N$.
Moreover, $\boc_1(G/N,\Ab)=0$ is a sophisticated reformulation
of Lemma~\ref{lem:cycsec}(b).
\end{proof}


\subsection{The Euler characteristic}
\label{ss:euler}
Let $G$ be a finite cyclic group, and let $\boB$
be the cohomological $G^\circ$-Mackey functor, $G^\circ=\{\triv,G\}$,
satisfying $\boB_G=\Z/|G|\cdot\Z$ and $\boB_{\triv}=0$.
Then $\boB$ has a projective resolution of length $3$ in $\cMF_{G^\circ}(\Zmod)$,
and one has natural isomorphisms
\begin{equation}
\label{eq:seccoh3}
\begin{aligned}
\bok^0(G,\argu)&\simeq\Ext^0(\boB,\argu),&\bok^1(G,\argu)&\simeq\Ext^1(\boB,\argu),& \\
\boc_1(G,\argu)&\simeq\Ext^2(\boB,\argu),&\boc_0(G,\argu)&=\Ext^3(\boB,\argu),
\end{aligned}
\end{equation}
where $\Ext^\bullet(\argu,\argu)$ denote the right derived functors of the homomorphisms functor
in $\cMF_{G^\circ}(\Zmod)$.
In particular, if
$0\rightarrow\boX\to\boY\to\boZ\to 0$ is a short exact sequence in $\cMF_{G^\circ}(\Zmod)$, then
one has a 12-term exact sequence
\begin{equation}
\label{eq:12term}
\xymatrix{
0\ar[r] & \bok^0(G,\boX)\ar[r] & \bok^0(G,\boY)\ar[r] & \bok^0(G,\boZ)\ar`r[d]`[l] `[dlll] `[dll] [dll] & \\
& \bok^1(G,\boX)\ar[r] & \bok^1(G,\boY)\ar[r] & \bok^1(G,\boZ)\ar`r[d]`[l] `[dlll] `[dll] [dll] &\\
& \boc^1(G,\boX)\ar[r] & \boc^1(G,\boY)\ar[r] & \boc^1(G,\boZ)\ar`r[d]`[l] `[dlll] `[dll] [dll] &\\
& \boc^0(G,\boX)\ar[r] & \boc^0(G,\boY)\ar[r] & \boc^0(G,\boZ)\ar[r] & 0}
\end{equation}
(cf. \cite[\S 4.1]{blth:cyc}). For a cohomological $G^\circ$-Mackey functor $\boX$
with values in the category of finitely generated $\Z$-modules (resp. $\Z_p$-modules)
one defines the {\it Euler characteristic} $\chi_G(\boX)$ of $\boX$ by
\begin{equation}
\label{eq:euler}
\chi_G(\boX)=\frac{|\bok^0(G,\boX)|\cdot |\boc_1(G,\boX)|}{|\bok^1(G,\boX)|\cdot|\boc_0(G,\boX)|}.
\end{equation}
Thus from \eqref{eq:12term} one concludes that for a short exact sequence
$0\rightarrow\boX\to\boY\to\boZ\to 0$ of cohomological $G^\circ$-Mackey functors
with values in the category of finitely generated $\Z$-modules (resp. $\Z_p$-modules)
one has
\begin{equation}
\label{eq:eulermult}
\chi_G(\boY)=\chi_G(\boX)\cdot\chi_G(\boZ).
\end{equation}
From Proposition~\ref{prop:h90} one concludes the following (cf. \eqref{eq:trat}).

\begin{prop}
\label{prop:euler}
Let $G$ be a finitely generated (pro-$p$) group, and let 
$N$ be a (closed) co-cyclic normal subgroup of finite index in $G$.
Then 
\begin{equation}
\label{eq:eucc}
\chi_{G/N}(\Ab)=\frac{\rho(G/N)}{|G:N|}.
\end{equation}
\end{prop}


\subsection{The Herbrand quotient}
\label{ss:herb}
As we have done before we will treat two cases simultaneously. 
We will either assume that $G$ is a finite cyclic group and $\Z_\bullet=\Z$,
or that $G$ is a finite cyclic $p$-group and $\Z_\bullet=\Z_p$.
If the second case applies, we will just add in parenthesis (...) the additional
hypothesis one has to make.

Let $G$ be a finite cyclic (p-)group, and let $M$ be a finitely generated
$\ZbG$-module. The rational number
\begin{equation}
\label{eq:herb1}
h(G,M)=\frac{|\hH^0(G,M)|}{|\hH^{-1}(G,M)|},
\end{equation}
is called the {\it Herbrand quotient} of the $\ZbG$-module $M$.
In particular, if $\boX$ is a cohomological $G^\circ$-Mackey functor
with values in the category of finitely generated $\Z_\bullet$-modules
the 6-term exact sequence \eqref{eq:6term}
implies that
\begin{equation}
\label{eq:euler2}
\chi_G(\boX)=\frac{|\bok^0(G,\boX)|\cdot |\boc_1(G,\boX)|}{|\bok^1(G,\boX)|\cdot|\boc_0(G,\boX)|}=h(G,\boX_{\triv})^{-1}.
\end{equation}
Let $\Q_\bullet$ denote the quotient field of $\Z_\bullet$.
For a finitely generated $\Z_\bullet$-module $M$ we put $M_{\Q_\bullet}=M\otimes_{\Z_\bullet}\Q_\bullet$.
The Herbrand quotient has the following well known properties.

\begin{prop}
\label{prop:herb}
Let $G$ be a finite cyclic ($p$-)group, and let $M$ be a finitely generated
$\ZbG$-module.
\begin{itemize}
\item[(a)] If $0\to A\to B\to C\to 0$ is a short exact sequence of 
$\ZbG$-modules, then 
\begin{equation}
\label{eq:multh}
h(G,B)=h(G,A)\cdot h(G,C).
\end{equation}
\item[(b)] If $M$ is a finite $\ZbG$-module, then $h(G,M)=1$.
\item[(c)] If $G$ is of prime order $p$, then one has
\begin{equation}
\label{eq:logh}
\dim_{\Q_\bullet}(M_{\Q_\bullet})=p\cdot\dim_{\Q_\bullet}(M^G_{\Q_\bullet})+(1-p)\cdot\log_p(h(G,M)).
\end{equation}
\end{itemize}
\end{prop}

\begin{proof}
For (a) and (b) see \cite[Chap.~IV, Thm.~7.3]{neu:alg}.
Part (c) can be found already in \cite[Chap.~IV.7, Ex.~3]{neu:alg}, but for the convenience
of the reader we give an alternative (and maybe simplier) proof here.
By (b), we may assume that $M$ is a torsion free $\Z_\bullet$-module.
By applying $\argu\otimes_{\Z_\bullet}\Z_p$ if necessary, we may assume that
$M$ is a $\Z_p[G]$-lattice. Hence, by a theorem of F.-E.~Diederichsen (cf. \cite[Thm.~34:31]{cr:mr1}, \cite{died:cp}),
there exist non-negative integers $r$, $s$ and $t$ such that
\begin{equation}
\label{eq:died}
M\simeq \Z_p^r\oplus \Omega^s\oplus\Z_p[G]^t,
\end{equation}
where $\Omega=\kernel(\Z_p[G]\to\Z_p)$ is the augmentation ideal, i.e., 
$\log_p(h(G,M))=r-s$. Thus from the equations
\begin{equation}
\label{eq:died2}
\begin{aligned}
\dim_{\Q_\bullet}(M_{\Q_\bullet})&=r+s\cdot (p-1)+t\cdot p,\\
\dim_{\Q_\bullet}(M_{\Q_\bullet}^G)&=r+t,
\end{aligned}
\end{equation}
one deduces the claim.
\end{proof}

\section{Conclusions}
\label{s:conc}

\subsection{Theorem B and Theorem B$^\prime$}
\label{ss:thmb}
{\it Proof.}
Let $G$ be a finitely generated (pro-$p$) group,
and let $N$ be a (closed) co-cyclic normal subgroup
of finite index. We assume further that $H^{\ab}$ is torsion
free for any subgroup $H$ of $G$ satisfying $N\subseteq H\subseteq G$.
Then, by hypothesis, $N^{\ab}$ is a $\Z_\bullet[G/N]$-lattice, but also a
$\Z_\bullet[H/N]$-lattice for
any subgroup $H$ of $G$ satisfying $N\subseteq H\subseteq G$.
As $\tk(H/N)=\kernel(i_{H,N}^{\Ab})=\bok^0(H/N,\Ab)$ is a finite
subgroup of $H^{\ab}$, one has $\tk(H/N)=0$ (cf. \S\ref{ss:seccoh}).
By the group theoretical version of Hilbert's theorem 90 (cf. Thm.~A), one has $\boc_1(H/N,\Ab)=0$
(cf. Prop.~\ref{prop:h90}). Hence from the 6-term exact sequence \eqref{eq:6term},
one concludes that $\hH^{-1}(H/N,N^{\ab})=0$. As $H/N$ is cyclic, Tate cohomology $\hH^\bullet(H/N,\argu)$
is periodic of period $2$ (cf. \cite[Chap.~VI.9, Ex.~9.2]{brown:coh}), i.e., one has $H^1(H/N,N^{\ab})=0$ for all subgroups $H/N$ of 
$G/N$. Now we have to consider the discrete case and the pro-$p$ case separetely.

If $G$ is discrete, the $\Z[G/N]$-lattice $N^{\ab}$ satisfies $H^1(H/N,N^{\ab})=0$ for all subgroups $H/N$ of $G/N$.
Thus, by \cite[Thm.~1.5]{enmi:int1}, $N^{\ab}$ is a pseudo $\Z[G/N]$-permutation module completing the proof of Theorem~B.

If $G$ is a pro-$p$ group, $N^{\ab}$ is a $\Z_p[G/N]$-lattice satifying $H^1(H/N,N^{\ab})=0$
for all subgroups $H/N$ of $G/N$. Hence, by \cite[Thm.~A]{blth:cyc}, $N^{\ab}$ is a 
$\Z_p[G/N]$-permutation module showing Theorem~B$^\prime$.
\hfill\qedsymbol

\subsection{Theorem C}
\label{ss:thmC}
{\it Proof.}
By hypthesis, the cohomological $(G/N)^\sharp$-Mackey functor
$\Ab$ has values in the category of finite abelian groups.
Hence $\chi_{G/N}(\Ab)=1$ (cf. \eqref{eq:euler} and Prop.~\ref{prop:herb}(b)).
The conclusion then follows from Proposition~\ref{prop:euler}.
\hfill\qedsymbol
\vskip6pt

Theorem~C can be easily translated in the original context of Hilbert's theorem~94 (cf. \cite[Thm.~94]{hil:alg}).
For a {\it number field} $K$ and a set of places $S\subset\caV(K)$
containing all infinite places we denote by $\caO^S(K)$ the {\it ring of $S$-integers},
and by $\eucl(\caO^S(K))$ its {\it ideal class group}. Thus, if $L/K$ is a finite Galois extension
of number fields with {\it Galois group} $G=\Gal(L/K)$ one has a canonical map
\begin{equation}
\label{eq:cap1}
i^S_{K,L}\colon \eucl(\caO^S(K))\longrightarrow\eucl(\caO^S(L))^G.
\end{equation}
Moreover, $\kernel(i^S_{K,L})$ is called the
{\it $S$-capitulation kernel}, and 
$\coker(i^S_{K,L})$ the {\it $S$-capitulation cokernel}\footnote{Our definition here might differ from the
definition used in \cite{ngu:cocap}.}
of the finite Galois extension $L/K$.
Theorem~C implies the following strong form of Hilbert's theorem 94.

\begin{thm}[Hilbert's theorem 94]
\label{thm:h94}
Let $K$ be a number field, let $S\subset\caV(K)$ be a finite
set of places containing all infinite places, and let $L/K$ be a cyclic
Galois extension, which is unramified at all places outset of $S$ and
completely split at all places in $S$. 
Then
\begin{equation}
\label{eq:cap}
|\kernel(i_{K,L}^S)|=|L:K|\cdot |\coker(i_{K,L}^S)|.
\end{equation} 
In particular, $|L:K|$ divides $|\kernel(i_{K,L}^S)|$.
\end{thm}

\begin{proof}
Let $\bar{K}/K$ be an algebraic closure of $K$ containing $L$, and let
$K^S/K$ be the maximal extension of $K$ which is unramified outside $S$ and
completely split for all places in $S$. In particular, $K^S/K$ is a Galois extension.
Let $G=\Gal(K^S/K)$ and let $U=G_L$ denote the pointwise stabilizer of $L$ in $G$.
Then, by class field theory and the finiteness of the class number, $G/[U,U]$ is a finite group, 
$G^{\ab}\simeq \cl(\caO^S(K))$
and $U^\ab\simeq \cl(\caO^S(L))$. The claim then follows from Theorem~C
and the commutativity
of the diagram
\begin{equation}
\label{eq:commdia}
\xymatrix{
\eucl(\caO^S(K))\ar[r]^{i_{K,L}^S}\ar[d]&\eucl(\caO^S(L))\ar[d]\\
G^{\ab}\ar[r]^{i_{G,U}}&U^{\ab}
}
\end{equation}
by the Artin reciprocity law (cf. \cite[Chap.~IV, Thm. 5.5]{neu:alg}).
\end{proof}

Let $G$ be a finite group, and let $N$ be a normal subgroup of $G$
such that $G/N$ is abelian. By H.~Suzuki's theorem, it is well known that
$|G:N|$ divides the order of $\kernel(i_{G,N}\colon G^{\ab}\to N^{\ab})$ (cf. \cite{suz:h94}).
Although Theorem~\ref{thm:h94} gives the answer in the case when $N$ is co-cyclic in $G$,
the following question remains open.

\begin{ques}
\label{q:1}
What is the value of the Hilbert-Suzuki multiplier 
\begin{equation}
\label{eq:hilsuz}
s_{G,N}=\frac{|\kernel(i_{G,N})|}{|G:N|}?
\end{equation}
\end{ques}

\subsection{Theorem D}
\label{ss:thmD}
{\it Proof.} Let $N$ be a (closed) co-cyclic normal subgroup
which of prime index $p$ in the finitely generated group $G$.
By \eqref{eq:euler} and Proposition~\ref{prop:h90}, one has 
\begin{equation}
\label{eq:herbD}
h(G/N,N^{\ab})=p\cdot\frac{|\tc(G/N)|}{|\tk(G/N)|}=p\cdot\rho(G/N)^{-1},
\end{equation}
and thus $\log_p(h(G,N^{\ab}))=1-\log_p(\rho(G/N))$. 
As $\bok^1(G/N,\Ab)$ is a finite $p$-group, one has 
an isomorphism of $\Q_\bullet$-vector spaces
$(N^{\ab})^G_{\Q_\bullet}\simeq G^{\ab}_{\Q_\bullet}$.
Hence, by \eqref{eq:logh}
\begin{equation}
\label{eq:gensch}
\dim_{\Q_\bullet}(N^{\ab}_{\Q_\bullet})=
p\cdot \dim_{\Q_\bullet}(G^{\ab}_{\Q_\bullet})+(1-p)(1-\log_p(\rho(G/N)))
\end{equation}
and hence the claim.
\hfill\qedsymbol

Let $G$ be a pro-$p$ group. 
A pair of open subgroups $(U,V)$, $V\subseteq U$, satisfying $|U:V|=p$
will be called an {\it open $p$-section}. In particular, $V$ is normal in $U$.
As a consequence of \eqref{eq:ratfor} one obtains the following
identification of the global tf-rank for finitely generated pro-$p$ groups.

\begin{cor}
\label{cor:gtf}
Let $G$ be a finitely generated pro-$p$ group of global tf-rank.
Then 
\begin{equation}
\gtf(G)=1-\log_p(\rho(U/V))\geq 0.
\end{equation}
for any open $p$-section $(U,V)$ of $G$.
\end{cor}

\subsection{Pro-$p$ groups of strict cohomological dimension less or equal to $2$}
\label{ss:scd2}

The definition of a cohomological Mackey functor can be easily extended to 
a profinite group $G$ (cf. \cite[\S 3]{thw:prdim}).

A cohomological $G^\sharp$-Mackey functor $\boX$ is said to be {\it $i$-injective}, if
the map $i^\boX_{U,V}\colon\boX_U\to\boX_V$ is injective
for all open subgroups $U$, $V$ of $G$, $V\subseteq U$.
An $i$-injective cohomological $G^\sharp$-Mackey functor $\boX$
is said to be {\it of type $H^0$} (or {\it to satisfy Galois descent}) if
$\bok^1(U/V,\boX)=0$ (cf. (\ref{eq:seccoh1})) for all open subgroups $U$, $V$ of $G$, $V$ normal in $U$.
A cohomological $G^\sharp$-Mackey functor of type $H^0$ is said to have the {\it Hilbert 90 property}
if $H^1(U/V,\boX_V)=0$ for all open subgroups $U$, $V$ of $G$, $V$ normal in $U$.

The following theorem extends A.~Brumer's characterisation of pro-$p$ groups
of strict cohomological dimension less or equal to $2$ (cf. \cite[Thm.~6.1]{brum:pseudo}).

\begin{thm}
\label{thm:scd2}
Let $G$ be a pro-$p$ group. Then the following are equivalent.
\begin{itemize}
\item[(i)] $\scd_p(G)\leq 2$;
\item[(ii)] the cohomological $G^\sharp$-Mackey functor $\Ab$ is of type $H^0$;
\item[(iii)] the cohomological $G^\sharp$-Mackey functor $\Ab$ has the Hilbert~90 property;
\item[(iv)] for every open $p$-section $(U,V)$ of $G$ one has
$\tk(U/V)=\tc(U/V)=0$.
\end{itemize}
\end{thm}

\begin{proof}
The equivalence (i)$\Rightarrow$(ii) is well known (cf. \cite[Chap.~III, Thm.~3.6.4]{nsw:cohn}),
and (iii) obviously implies (ii). 

Suppose that (ii) holds, and let $(U,V)$ be an open $p$-section of $G$.
In particular, $V$ is a co-cyclic open normal subgroup of $U$,
and Theorem A implies that $\boc_1(U/V,\Ab)=0$. Hence, as \[\kernel(i_{U,V})=\bok^0(U/V,\Ab)=0,\]
the 6-term exact sequence (cf. (\ref{eq:6term})) yields that $\hH^{-1}(U/V,V^{\ab})=0$.
As $U/V$ is cyclic and thus has periodic cohomology of period $2$, this implies that 
\[H^1(U/V,V^{\ab})=\hH^{-1}(U/V,V^{\ab})=0\]
for any open $p$-section of $G$. Let $U/W$ be an open normal section of $G$, i.e. 
$U$ and $W$ are open subgroups of $G$ and $W$ is normal in $U$. As $U/W$ is a finite $p$-group,
there exists a decreasing sequence of open subgroups $(W_k)_{0\leq k\leq n}$ of $U$ satisfying $W_0=U$, $W_n=W$ and
$|W_k:W_{k+1}|=p$. We claim that $H^1(U/W,W^{\ab})=0$. For proving this claim we proceed by induction on $n$.
The previously mentioned argument shows the claim for $n=1$. Assume that $n\geq 2$.
As (ii) holds, one has $(W^{\ab})^{W_1}\simeq W_1^{\ab}$. Hence the induction hypothesis implies that
$H^1(U/W_1,W^{\ab}_1)=0$ and $H^1(W_1/W,W^{\ab})=0$. Thus the 5-term sequence associated to
the Hochschild-Serre spectral sequence (cf. \cite[Thm.~2.4.1]{nsw:cohn})
yields that $H^1(U/W,W^{\ab})=0$ and (iii) holds.

The implication (ii)$\Rightarrow$(iv) is obvious.
Suppose that (iv) holds, and let $(U,W)$ be a section in $G$, i.e.,
$U$ and $W$ are open subgroups of $G$ and $W$ is contained in $U$.
If $W\not=U$, then $N_U(W)$ is strictly larger than $W$. From this property one concludes that
there exists a decreasing sequence of open subgroups $(W_k)_{0\leq k\leq n}$ of $U$ satisfying $W_0=U$, $W_n=W$ and
$|W_k:W_{k+1}|=p$. Moreover, $i_{U,W}=i_{W_{n-1},W}\circ\cdots\circ i_{U,W_1}$.
By hypothesis, one has \[\tk(W_k/W_{k+1})=\kernel(i_{W_k,W_{k+1}})=0,\]
i.e., $i_{W_k,W_{k+1}}$ is injective for $0\leq k\leq n-1$. Hence $i_{U,W}$ is injective.
If $W$ is normal in $U$, we may assume that the open subgroups $W_k$ are normal in $U$.
By hypothesis, one has that $i^\circ_{U,W_1}\colon U^{\ab}\to (W_1^{\ab})^{U/W_1}$
is an isomorphism. By induction, we may suppose that
$i^\circ_{W_1,W}\colon W_1^{\ab}\to (W^{\ab})^{W_1/W}$ is an isomorphism,
i.e., \[(i^\circ_{W_1,W})^{U/W_1}\colon (W_1^{\ab})^{U/W_1}\longrightarrow (W^{\ab})^{U/W}\]
is an isomorphism as well. Thus, as $i_{U,W}=(i^\circ_{W_1,W})^{U/W_1}\circ i^\circ_{U,W_1}$,
this yields (ii).
\end{proof}

Corollary~\ref{cor:gtf} and Theorem~\ref{thm:scd2} imply the following.

\begin{cor}
\label{cor:brum}
Let $G$ be a finitely generated pro-$p$ group of global tf-rank satisfying $\scd_p(G)\leq 2$. Then 
either
\begin{itemize}
\item[(i)] $G=\triv$; or
\item[(ii)] $G\simeq\Z_p$; or
\item[(iii)] $\ccd_p(G)=2$ and $\gtf(G)=1$.
\end{itemize}
\end{cor}

\begin{rem}
\label{rem:scd2}
Let $\Z_p^\times$ denote the multiplicative group of the $p$-adic integers, and let 
$\theta\colon \Z_p\to\Z_p^\times$ be a homomorphism of profinite groups with open image.
It is straightforward to verify that the semi-direct product $G_\theta=\Z_p\ltimes_\theta\Z_p$
is a $p$-adic analytic pro-$p$ group satisfying $\ccd_p(G_\theta)=\scd_p(G_\theta)=2$ and $\gtf(G_\theta)=1$. 
Moreover, any $p$-adic analytic group satisfying $\ccd_p(G_\theta)=\scd_p(G_\theta)=2$ and $\gtf(G_\theta)=1$
must be isomorphic to some $G_\theta$, 
\[\theta\in\left\{\,\alpha\in\Hom(\Z_p,\Z_p^\times)\mid\alpha\ \mathrm{cont. \ and \ open}\,\right\}.\]
Therefore, the following question arises.

\begin{ques}
\label{q:1}
Let $G$ be a finitely generated pro-$p$ group
of global tf-rank $1$ satisfying $\scd_p(G)\leq 2$.
Is $G$ necessarily $p$-adic analytic?
\end{ques}

In a private discussion with the second author A.~Jaikin-Zapirain asked
a similar question.

\begin{ques}[A.~Jaikin-Zapirain, 2012]
\label{q:2}
Let $F$ be a finitely generated free pro-$p$ group, 
and suppose that for some injective homomomorphism $\beta\colon \Z_p\to\Aut(F)$
the pro-$p$ group $G=\Z_p\ltimes_\beta F$ is of global tf-rank $1$.
Does this imply that $F$ is of rank $1$?
\end{ques}

An affirmitive answer to Question~\ref{q:2} would settle Question~\ref{q:1}
in some important particular cases.
\end{rem}

\section*{Acknowledgements}
In 1999, K.W.~Gruenberg gave a talk on {\it Transfer kernels} at the
{\it Mathematisches Institut der Universit\"at Freiburg}, Germany, in honour of O.H.~Kegel's 65$^{th}$ birthday.
This talk and the subsequent discussions with K.W.Gruenberg on various occasions were the motivation for the second author
to introduce and study the section cohomology groups of cohomological Mackey functors (cf. \S\ref{ss:seccoh}).

\bibliography{hilbert}
\bibliographystyle{amsplain}
\end{document}